\RequirePackage{amsmath}
\documentclass{llncs}

\usepackage[utf8]{inputenc}
\usepackage{latexsym}
\usepackage[T1]{fontenc}
\usepackage{amssymb}
\usepackage{stmaryrd}
\usepackage{marvosym}
\usepackage{mdwmath}
\usepackage{mdwtab}

\usepackage{graphicx}
\usepackage[cmyk]{xcolor}
\usepackage{tikz}

\usepackage{tabularx}
\usepackage{array}
\usepackage{eqparbox}
\usepackage{booktabs}

\usepackage{url}
\usepackage{hyperref}

\newcommand{\pc}{\mathbf{P}}

\newcommand{\pp}{\mathbb{P}}

\newcommand{\shg}{ higher granular operator space }

\newcommand{\vect}[2]{{#1_1,#1_2,\dotsc,#1_#2}}

\begin{document}

\setcounter{page}{1}
 
\title{Approximations from Anywhere and General Rough Sets }

\author{\textbf{A. Mani}}
\institute{Department of Pure Mathematics\\
University of Calcutta\\
9/1B, Jatin Bagchi Road\\
Kolkata(Calcutta)-700029, India\\
\email{$a.mani.cms@gmail.com$}\\
Homepage: \url{http://www.logicamani.in}}

\maketitle

\begin{abstract}
Not all approximations arise from information systems. The problem of fitting approximations, subjected to some rules (and related data), to information systems in a rough scheme of things is known as the \emph{inverse problem}. The inverse problem is more general than the duality (or abstract representation) problems and was introduced by the present author in her earlier papers. From the practical perspective, a few (as opposed to one) theoretical frameworks may be suitable for formulating the problem itself. \emph{Granular operator spaces} have been recently introduced and investigated by the present author in her recent work in the context of antichain based and dialectical semantics for general rough sets. The nature of the inverse problem is examined from number-theoretic and combinatorial perspectives in a higher order variant of granular operator spaces and some necessary conditions are proved. The results and the novel approach would be useful in a number of unsupervised and semi supervised learning contexts and algorithms.

\end{abstract}

\section{Introduction}

General rough set theory specifically targets information systems as the object of study in the sense that starting from information systems, approximations are defined and rough objects of various kinds are studied \cite{ZPB,gdu,AM240,lp2011,sk05,ppm2}. But the focus need not always be so. In duality problems, the problem is to generate the information system (to the extent possible) from semantic structures like algebras or topological algebras associated (see for example \cite{BC1,du,AM3,AM240,IGO2007,IGO2001}). Logico-algebraic and other semantic structures typically capture reasoning and processes in the earlier mentioned approach.

The concept of \emph{inverse problem} was introduced by the present author in \cite{AM3} and was subsequently refined in \cite{AM240}. In simple terms, the problem is a generalization of the duality problem which may be obtained by replacing the semantic structures with parts thereof. Thus the goal of the problem is to fit a given set of approximations and some semantics to a suitable rough process originating from an information system. Examples of approximations that are not rough in any sense are common in misjudgments and irrational reasoning guided by prejudice. 

The above simplification is obviously dense because it has not been formulated in a concrete setup. It also needs many clarifications at the theoretical level. At the theoretical level again a number of frameworks appear to be justified. These can be restricted further by practical considerations. All this will not be discussed in full detail in this paper (for reasons of space). Instead a specific minimalist framework called \emph{Higher Granular Operator Space} is proposed first and the problem is developed over it by the present author. All of the results proved are from a combinatorial perspective and on the basis of these results the central question can be answered in the negative in many cases.

\section{Background}

A relational system is a tuple of the form $\mathfrak{S} = \left\langle S, \vect{R}{n} \right\rangle$ with $S$ being a set and $R_i$ being predicates of arity $\nu_i$. The type of $\mathfrak{S}$ is $(\vect{\nu}{n})$.  If $\mathfrak{H} = \left\langle H, \vect{Q}{n} \right\rangle$ is another relational system of the same type and $\varphi :S \longmapsto H$ a map satisfying for each $i$, \[ (R_i a_1 a_2 \ldots a_{\nu_{i}} \longrightarrow R_i \varphi (a_1) \varphi (a_2)\ldots \varphi (a_{\nu_{i}})),\]
is a \emph{relational morphism} \cite{mal}. If $\varphi$ is also bijective, then it is referred to as a \emph{relational isomorphism}.

By an \emph{Information System} $\mathcal{I}$, is meant a structure of the form \[\mathcal{I}\,=\, \left\langle \mathbb{O},\, At,\, \{V_{a} :\, a\in At\},\, \{f_{a} :\, a\in At\}  \right\rangle \]
with $\mathbb{O}$, $At$ and $V_{a}$ being respectively sets of \emph{Objects}, \emph{Attributes} and \emph{Values} respectively. It is \emph{deterministic} (or complete) if for each $a\in At$, $f_{a}:\, O \,\longmapsto \, V_{a}$ is a map. It is said to be \emph{indeterministic} (or incomplete) if the valuation has the form $f_{a}:\, O\,\longmapsto \, \wp(V)$, where $V\,=\, \bigcup V_{a}$.
These two classes of information systems can be used to generate various types of relational, covering or relator spaces which in turn relate to approximations of different types and form a substantial part of the problems encountered in general rough set theories. One way of defining an indiscernibility relation $\sigma$ is as below:

For $x,\, y\,\in\, \mathbb{O} $ and $B\,\subseteq\, At $, $(x,\,y)\,\in\, \sigma $ if and only if $(\forall a\in B)\, \nu(a,\,x)\,=\, \nu (a,\, y) $. In this case $\sigma$ is an equivalence relation (see \cite{ZPB,BC1,du,BC2}). Lower and upper approximations, rough equalities are defined over it and topological algebraic semantics can be formulated over \emph{roughly equivalent} objects (or subsets of attributes) through extra operations. Duality theorems, proved for \emph{pre-rough} algebras defined in \cite{BC1}, are specifically for structures relation isomorphic to the approximation space $\left(\mathbb{O}, \sigma \right)$. This is also true of the representation results in \cite{AM3,du,IE2013}. But these are not for information systems - optimal concepts of \emph{isomorphic information systems} are considered by the present author in a forthcoming paper.

In fact in \cite{AM3}, it has been proved by the present author that 
\begin{theorem}
For every super rough algebra $S$, there exists an approximation space $X$ such that the super rough set algebra generated by $X$ is isomorphic to $S$.
\end{theorem}

In simple terms, \emph{granules} are the subsets (or objects) that generate approximations and \emph{granulations} are the collections of all such granules in the context. For more on what they might be the reader may refer to \cite{AM240,AM3930}. In this paper a variation of generalized granular operator spaces, introduced and studied by the present author in \cite{AM6999,AM9114,AM9699}, will serve as the primary framework for most considerations. For reference, related definitions are mentioned below.

\begin{definition}\label{gos}
A \emph{Granular Operator Space}\cite{AM6999} $S$ is a structure of the form $S\,=\, \left\langle \underline{S}, \mathcal{G}, l , u\right\rangle$ with $\underline{S}$ being a set, $\mathcal{G}$ an \emph{admissible granulation}(defined below) over $S$ and $l, u$ being operators $:\wp(\underline{S})\longmapsto \wp(\underline{S})$ ($\wp(\underline{S})$ denotes the power set of $\underline{S}$) satisfying the following ($\underline{S}$ is replaced with $S$ if clear from the context. \textsf{Lower and upper case alphabets may denote subsets} ):

\begin{align*}
a^l \subseteq a\,\&\,a^{ll} = a^l \,\&\, a^{u} \subseteq a^{uu}  \\
(a\subseteq b \longrightarrow a^l \subseteq b^l \,\&\,a^u \subseteq b^u)\\
\emptyset^l = \emptyset \,\&\,\emptyset^u = \emptyset \,\&\,\underline{S}^{l}\subseteq S \,\&\, \underline{S}^{u}\subseteq S.
\end{align*}

In the context of this definition, \emph{Admissible Granulations} are granulations $\mathcal{G}$ that satisfy the following three conditions ($t$ being a term operation formed from the set operations $\cup, \cap, ^c, 1, \emptyset$):

\begin{align*}
(\forall a \exists
b_{1},\ldots b_{r}\in \mathcal{G})\, t(b_{1},\,b_{2}, \ldots \,b_{r})=a^{l} \\
\tag{Weak RA, WRA} \mathrm{and}\: (\forall a)\,(\exists
b_{1},\,\ldots\,b_{r}\in \mathcal{G})\,t(b_{1},\,b_{2}, \ldots \,b_{r}) =
a^{u},\\
\tag{Lower Stability, LS}{(\forall b \in
\mathcal{G})(\forall {a\in \wp(\underline{S}) })\, ( b\subseteq a\,\longrightarrow\, b \subseteq a^{l}),}\\
\tag{Full Underlap, FU}{(\forall
a,\,b\in\mathcal{G})(\exists
z\in \wp(\underline{S}) )\, a\subset z,\,b \subset z\,\&\,z^{l} = z^{u} = z,}
\end{align*}
\end{definition}

\begin{flushleft}
\textbf{Remarks}: 
\end{flushleft}
\begin{itemize}
\item {The concept of admissible granulation was defined for \textsf{RYS} in \cite{AM240} using parthoods instead of set inclusion and relative to \textsf{RYS}, $\pc = \subseteq$, $\pp = \subset$. It should be noted that the minimal assumptions make this concept more general than the idea of granulation in the precision based granular computing paradigm (and complex granules) \cite{JS09,AM9501}. }
\item {\emph{The conditions defining admissible granulations mean that every approximation is somehow representable by granules in a set theoretic way, that granules are lower definite, and that all pairs of distinct granules are contained in definite objects}.}
\item {the term operation $t$ is intended to be defined over the power set Boolean algebra in standard algebraic sense (see \cite{AM960} for a detailed example).}
\end{itemize}

The concept of \emph{generalized granular operator spaces} has been introduced in \cite{AM9114,AM6900} as a proper generalization of that of granular operator spaces. The main difference is in the replacement of $\subset$ by arbitrary \emph{part of} ($\pc$) relations in the axioms of admissible granules and inclusion of $\pc$ in the signature of the structure.
\begin{definition}
A \emph{General Granular Operator Space} (\textsf{GSP}) $S$ is a structure of the form $S\,=\, \left\langle \underline{S}, \mathcal{G}, l , u, \pc \right\rangle$ with $\underline{S}$ being a set, $\mathcal{G}$ an \emph{admissible granulation}(defined below) over $S$, $l, u$ being operators $:\wp(\underline{S})\longmapsto \wp(\underline{S})$ and $\pc$ being a definable binary generalized transitive predicate (for parthood) on $\wp(\underline{S})$ satisfying the same conditions as in Def.\ref{gos} except for those on admissible granulations (Generalized transitivity can be any proper nontrivial generalization of parthood (see \cite{AM9501}). $\pp$ is  proper parthood (defined via $\pp ab$ iff $\pc ab \,\&\,\neg \pc ba$) and $t$ is a term operation formed from set operations):

\begin{align*}
(\forall x \exists
y_{1},\ldots y_{r}\in \mathcal{G})\, t(y_{1},\,y_{2}, \ldots \,y_{r})=x^{l} \\
\tag{Weak RA, WRA} \mathrm{and}\: (\forall x)\,(\exists
y_{1},\,\ldots\,y_{r}\in \mathcal{G})\,t(y_{1},\,y_{2}, \ldots \,y_{r}) =
x^{u},\\
\tag{Lower Stability, LS}{(\forall y \in
\mathcal{G})(\forall {x\in \wp(\underline{S}) })\, ( \pc yx\,\longrightarrow\, \pc yx^{l}),}\\
\tag{Full Underlap, FU}{(\forall
x,\,y\in\mathcal{G})(\exists
z\in \wp(\underline{S}) )\, \pp xz,\,\&\,\pp yz\,\&\,z^{l} = z^{u} = z,}
\end{align*}
\end{definition}

\subsection{Finite Posets}\label{wth}

Let $S$ be a finite poset with $\#(S) = n < \infty$. The following concepts and notations will be used in this paper:
\begin{itemize}
\item {If $\mathfrak{F}$ is a collection of subsets $\{X_i\}_{i\in J}$ of a set $X$, then a \emph{system of distinct representatives} \textsf{SDR} for $\mathfrak{F}$ is a set $\{x_i ; i\in J\}$ of distinct elements satisfying $(\forall i \in J) x_i \in X_i$. Chains are subsets of a poset in which any two elements are comparable. Singletons are both chains and antichains.}
\item {For $a, b\in S$, $a \prec b$ shall be an abbreviation for $b$ covering $a$ from above (that is $a < b$ and $(a\leq c \leq b \longrightarrow c= a \text{ or } c= b)$). $c(S)$ shall be the number of covering pairs in $S$.}
\item {A \emph{chain cover} of a finite poset $S$ is a collection $\mathcal{C}$ of chains in $S$ satisfying $\cup \mathcal{C} = S$. It is disjoint if the chains in the cover are pairwise disjoint.}
\item {$S$ has finite width $w$ if and only if it can be partitioned into $w$ number of chains, but not less.}
\end{itemize}

The following results are well known:
\begin{theorem}
\begin{enumerate}
\item {A collection of subsets $\mathfrak{F}$ of a finite set $S$ with $\#(\mathfrak{F})= r$ has an SDR if and only if for any $1 \leq k \leq r $,the union of any $k$ members of $\mathfrak{F}$ has size at least $k$, that is 
\[(\forall{X_1,\ldots , X_k \in \mathfrak{F}}) \, k \leq \#(\cup X_i).\] }
\item {Every finite poset $S$ has a disjoint chain cover of width $w = width(S)$.}
\item {If $X$ is a partially ordered set with longest chains of length $r$ and if it can be partitioned into $k$ number of antichains then $r\leq k$.}
\item {If $X$ is a finite poset with $k$ elements in its largest antichain, then a chain decomposition of $X$ must contain at least $k$ chains. }
\end{enumerate} 
\end{theorem}

Proofs of the assertions can be found in \cite{JGC2009,DP2002} for example. To prove the third, start from a chain decomposition and recursively extract the minimal elements from it to form $r$ number of antichains. The fourth assertion is proved by induction on the size of $X$ across many possibilities.

\section{Semantic Framework}

It is more convenient to use only sets and subsets in the formalism as these are the kinds of objects that may be observed by agents and such a formalism would be more suited for reformulation in formal languages. This justifies the severe variation defined below in stages:

\begin{definition}\label{rosp}
A \emph{Higher Rough Operator Space} $\mathbb{S}$ shall be a structure of the form $\mathbb{S}\,=\, \left\langle \underline{\mathbb{S}}, l , u, \leq , \bot, \top \right\rangle$ with $\underline{\mathbb{S}}$ being a set,  and $l, u$ being operators $:\underline{\mathbb{S}}\longmapsto \underline{\mathbb{S}}$ satisfying the following ($\underline{\mathbb{S}}$ is replaced with $\mathbb{S}$ if clear from the context. ):

\begin{align*}
(\forall a \in \mathbb{S})\, a^l \leq a\,\&\,a^{ll} = a^l \,\&\, a^{u}  \leq  a^{uu}  \\
(\forall a, b \in \mathbb{S}) (a\leq b \longrightarrow a^l \leq b^l \,\&\,a^u \leq b^u)\\
 \bot^l =  \bot \,\&\, \bot^u =  \bot \,\&\, \top^{l}\leq \top \,\&\,  \top^{u}\leq \top \\
 (\forall a \in \mathbb{S})\, \bot \leq a \leq \top\\
 \mathbb{S} \text{ is a bounded poset}.
\end{align*}
\end{definition}

\begin{definition}\label{hos}
A \emph{Higher Granular Operator Space} (\textsf{SHG}) $\mathbb{S}$ shall be a structure of the form $\mathbb{S}\,=\, \left\langle \underline{\mathbb{S}}, \mathcal{G}, l , u, \leq, \vee, \wedge, \bot, \top \right\rangle$ with $\underline{\mathbb{S}}$ being a set, $\mathcal{G}$ an \emph{admissible granulation}(defined below) for $\mathbb{S}$ and $l, u$ being operators $:\underline{\mathbb{S}}\longmapsto \underline{\mathbb{S}}$ satisfying the following ($\underline{\mathbb{S}}$ is replaced with $\mathbb{S}$ if clear from the context. ):

\begin{align*}
 (\mathbb{S}, \vee, \wedge, \bot, \top )\text{ is a bounded lattice}\\
 \leq \text{ is the lattice order} \\
(\forall a \in \mathbb{S})\, a^l \leq a\,\&\,a^{ll} = a^l \,\&\, a^{u}  \leq  a^{uu}  \\
(\forall a, b \in \mathbb{S}) (a\leq b \longrightarrow a^l \leq b^l \,\&\,a^u \leq b^u)\\
 \bot^l =  \bot \,\&\, \bot^u =  \bot \,\&\, \top^{l}\leq \top \,\&\,  \top^{u}\leq \top \\
 (\forall a \in \mathbb{S})\, \bot \leq a \leq \top
\end{align*}

$\pc ab$ if and only if $a \leq b$ in the following three conditions. Further $\pp$ is  proper parthood (defined via $\pp ab$ iff $\pc ab \,\&\,\neg \pc ba$) and $t$ is a term operation formed from the lattice operations):

\begin{align*}
(\forall x \exists y_{1},\ldots y_{r}\in \mathcal{G})\, t(y_{1},\,y_{2}, \ldots \,y_{r})=x^{l} \\
\tag{Weak RA, WRA} \mathrm{and}\: (\forall x)\,(\exists
y_{1},\,\ldots\,y_{r}\in \mathcal{G})\,t(y_{1},\,y_{2}, \ldots \,y_{r}) =
x^{u},\\
\tag{Lower Stability, LS}{(\forall y \in \mathcal{G})(\forall x\in \underline{\mathbb{S}})\, (\pc yx\,\longrightarrow\, \pc yx^{l}),}\\
\tag{Full Underlap, FU}{(\forall x,\,y\in\mathcal{G})(\exists z\in \underline{\mathbb{S}} )\, \pp xz,\,\&\,\pp yz\,\&\,z^{l} = z^{u} = z}
\end{align*}
\end{definition}

\begin{definition}
An element $x\in\mathbb{S}$ will be said to be \emph{lower definite} (resp. \emph{upper definite}) if and only if $x^l = x$ (resp. $x^u = x$) and \emph{definite}, when it is both lower and upper definite. $x\in \mathbb{S}$ will also be said to be \emph{weakly upper definite} (resp \emph{weakly definite}) if and only if $ x^u = x^{uu} $ (resp $ x^u = x^{uu} \,\&\, x^l =x$ ). Any one of these five concepts may be chosen as a concept of \emph{crispness}. 
\end{definition}

The following concepts of \emph{rough objects} have been either considered in the literature (see \cite{AM240}) or are reasonable concepts:
\begin{itemize}
\item {$x\in \mathbb{S}$ is a lower rough object if and only if $\neg (x^l = x) $. }
\item {$x\in \mathbb{S}$ is a upper rough object if and only if $\neg (x = x^u) $. }
\item {$x\in \mathbb{S}$ is a weakly upper rough object if and only if $\neg (x^u = x^{uu}) $. }
\item {$x\in \mathbb{S}$ is a rough object if and only if $\neg (x^l = x^u) $. }
\item {\emph{Any pair of definite elements} of the form $(a , b)$ satisfying $a < b $}
\item {\emph{Any distinct pair of elements} of the form $(x^l ,x^u)$.}
\item {Elements in an \emph{interval of the form} $(x^l, x^u)$.}
\item {Elements in an \emph{interval of the form} $(a, b)$ satisfying $a\leq b$ with $a, b$ being definite elements.}
\item {A \emph{non-definite element in a RYS}(see \cite{AM240}), that is an $x$ satisfying $\neg \pc x^u x^l   $}
\end{itemize}

All of the above concepts of a rough object except for the last are directly usable in a \shg. Importantly, \emph{most of the results proved in this paper can hold for many choices of concepts of roughness and crispness. The reader is free to choose suitable combinations from the $40$ possibilities}. 

\begin{example}[No Information Tables]
It should be easy to see that most examples of general rough sets derived from information tables (and involving granules and granulations) can be read as \shg . So a nontrivial example of a \shg that has not been derived from an information system is presented below:

Suppose agent X wants to complete a task and this task is likely to involve the use of a number of tools. X thinks tool-1 \textsf{suffices for} the task that a tool-2 \textsf{is not suited for} the purpose and that tool-3 is \textsf{better suited than} tool-1 for the same task. X also believes that tool-4 \textsf{is as suitable as} tool-1 for the task and that tool-5 \textsf{provides more than what is necessary for} the task.  X thinks similarly about other tools  but not much is known about the consistency of the information. X has a large repository of tools and
limited knowledge about tools and their suitability for different purposes, and at the same time X might be knowing more about difficulty of tasks that in turn require better tools of different kinds. 

Suppose also that similar heuristics are available about other similar tasks. The reasoning of the agent in the situation can be recast in terms of lower, upper approximations and generalized equality and questions of interest include those relating to the agent's understanding of the features of tools, their appropriate usage contexts and whether the person thinks rationally.

To see this it should be noted that the key predicates in the context are as below:
\begin{itemize}
\item {\textsf{suffices for} can be read as \emph{includes potential lower approximation of} a right tool for the task. }
\item {\textsf{is not suited for} can be read as \emph{is neither a lower or upper approximation of} any of the right tools for the task. }
\item {\textsf{better suited than} can be read as \emph{potential rough inclusion },}
\item {\textsf{is as suitable as} can be read as  \emph{potential rough equality} and}
\item {\textsf{provides more than what is necessary for} is for \emph{upper approximation} of a right tool for the task. }
\end{itemize} 
\end{example}
\begin{example}[Number of Objects]
Often in the design, implementation and analysis of surveys (in the social sciences in particular), a number of intrusive assumptions on the sample are done and preconceived ideas about the population may influence survey design. Some assumptions that ensure that the sample is representative are obviously good, but as statistical methods are often abused \cite{MHRL} a minimal approach can help in preventing errors. The idea of samples being representative translates into number of non crisp objects being at least above a certain number and below a certain number. There are also situations (as when prior information is not available or ideas of representative samples are unclear) when such bounds may not be definable or of limited interest.
\end{example}
\begin{example}[Non-Rough Approximations]
Suppose $X_1, \ldots X_{24}$ are $24$ colors defined by distinct frequencies and suppose the weak sensors at disposal can identify $3$ of them as crisp colors. If it is required that the other $21$ colors be approximated as $9$ rough objects, then such a classification would not be possible in a rough scheme of things as at most three distinct pairs of crisp objects are possible. Note that using intervals of frequencies, tolerances can be defined on the set. But under the numeric restriction, $9$ rough objects would not be possible.  
\end{example}

\subsection{Minimal Assumptions}\label{sf}

For the considerations of the following sections on distribution of rough objects and on counting to be valid, a minimal set of assumptions are necessary. These will be followed unless indicated otherwise:

\begin{description}
\item[F1]{$\mathbb{S} \text{ is a \shg }$.}
\item[F2]{$ \# (\mathbb{S}) = n < \infty .$}
\item[C1]{$ C \subseteq \mathbb{S} \text{ is the set of crisp objects}.$}
\item[C2]{$ \# (C) = k .$}
\item[R1]{$ R \subset \mathbb{S} \text{ is the set of rough objects not necessarily defined as in the above}.$}
\item[R2]{$ R \cup C = \mathbb{S}$}
\item[R3]{$ \text{ there exists a map } \varphi : R \longmapsto C^{2} .$}
\item[RC1]{$ R\cap C = \emptyset .$}
\item[RC2]{$ (\forall x\in R)(\exists a, b\in C) \varphi (x) = (a, b) \,\& \, a\subset b.$} 
\end{description}

Note that no further assumptions are made about the nature of $\varphi(x)$. It is not required that 
$\varphi (x) = (a, b) \, \&\, x^l = a \, \&\, x^u = b$, though this happens often. 

The set of crisp objects is necessarily partially ordered. In specific cases, this order may be a lattice, distributive, relatively complemented or Boolean order. Naturally the combinatorial features associated with \shg depend on the nature of the partial order. This results in situations that are way more involved than the situation encoded by the following simple proposition.

\begin{proposition}
Under all of the above assumptions, for a fixed value of $\# (\mathbb{S}) = n $ and $\#(C)=k$, $R$ must be representable by a finite subset $K \subseteq C^2 \setminus \Delta_C$, $\Delta_C$ being the diagonal in $C^2$.   
\end{proposition}

The two most extreme cases of the ordering of the set $C$ of crisp objects correspond to $C$ forming a chain and $C\setminus \{\bot,\top \}$ forming an anti-chain. Numeric measures for these distributions have been defined for these in \cite{AM9114} by the present author. The measure gives an idea of the extent of distribution of non crisp objects over the distribution of the crisp objects and it has also been shown that such measures do not provide reasonable comparisons across diverse contexts.

\section{Pre-Well Distribution of Objects over Chains: PWC}

\begin{definition}
A distribution of rough objects relative to a chain of crisp objects $C$ will be said to be a \emph{Pre-Well Distribution of Objects over Chains} if the minimal assumptions (Subsection.\ref{sf}) (without the condition \textsf{RC1}) and the following three 
conditions hold: 

\begin{enumerate}
 \item {$C$ forms a chain under inclusion order.}
 \item {$\varphi$ is a surjection.}
 \item {Pairs of the form $(x, x)$, with $x$ being a crisp object, also correspond to rough objects.}
\end{enumerate}
\end{definition}

Though the variant is intended as an abstract reference case where the idea of crispness is expressed subliminally, there are very relevant practical contexts for it (see Example. \ref{exa1}). It should also be noted that this interpretation is not compatible with the interval way of representing rough objects without additional tweaking. 

\begin{theorem}
Under the above assumptions, the number of crisp objects is related to the total number of objects by the formula: \[k \stackrel{i}{=} \dfrac{(1+4n)^\frac{1}{2} -1 }{2}.\] 
In the formula $\stackrel{i}{=}$ is to be read as \emph{if the right hand side (RHS) is an integer then the left hand side is the same as RHS.}  
\end{theorem}

\begin{proof}
\begin{itemize}
\item {Clearly the number of rough objects is $n - k$ .}
\item {By the nature of the surjection $n - k$ maps to $k^2$  pairs of crisp objects.}
\item {So $n - k = k^2$.}
\item {So integral values of $\dfrac{(1+4n)^\frac{1}{2} -1 }{2}$ will work.}  
\end{itemize}
\qed
\end{proof}

This result is associated with the distribution of odd square integers of the form $4n +1$ which in turn should necessarily be of the form $4(p^2 +p) +1$ (p being any integer). The requirement that these be perfect squares causes the distribution of crisp objects to be very sparse with increasing values of $n$.  The number of rough objects between two successive crisp objects increases in a linear way, but this is a misleading aspect. These are illustrated in the graphs \textsf{Fig}.\ref{fig1} and \textsf{Fig}.\ref{fig2}.

\begin{figure*}[htb]
\centering
\includegraphics[width=11.7cm]{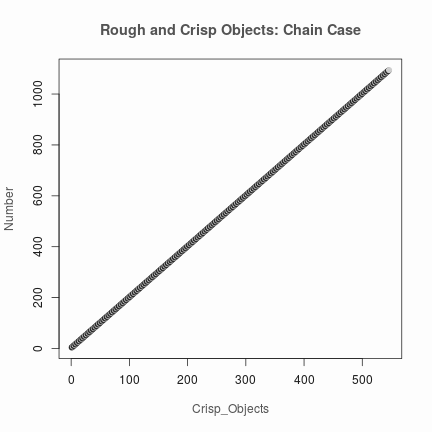}
\caption{Rough Objects Between Crisp Objects: Special Chain Case}\label{fig1}
\end{figure*}

\begin{figure*}[htb]
\centering
\includegraphics[width=11.7cm]{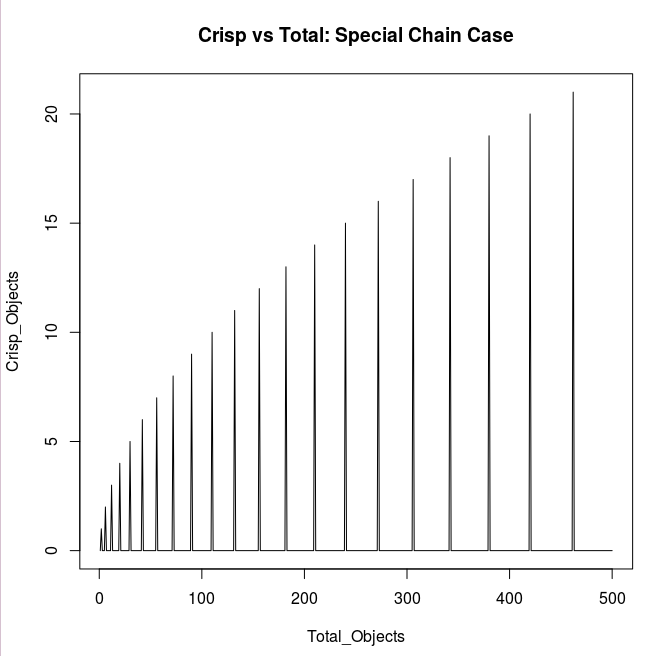}
\caption{Values of n and k: Special Chain Case}\label{fig2}
\end{figure*}

\begin{example}\label{exa1}
This example has the form of a narrative that gets progressively complex.

Suppose Alice wants to purchase a laptop from an on line store for electronics. Then she is likely to be confronted by a large number of models and offers from different manufacturers and sellers. Suppose also that the she is willing to spend less than \EURdig $x$ and is pretty open to considering a number of models. This can happen, for example, when she is just looking for a laptop with enough computing power for her programming tasks.   

This situation may appear to have originated from information tables with complex rules in columns for decisions and preferences. Such tables are not information systems in the proper sense. Computing power for one thing is a context dependent function of CPU cache memories, number of cores, CPU frequency, RAM, architecture of chipset, and other factors like type of hard disk storage. 

\begin{proposition}
The set of laptops $\mathbb{S}$ that are priced less than \EURdig $x$ can be totally quasi ordered. 
\end{proposition}
\begin{proof}
Suppose $\prec$ is the relation defined according to $a \prec b$ if and only if price of laptop $a$ is less than or equal to that of laptop $b$. Then it is easy to see that $\prec$ is a reflexive and transitive relation.
If two different laptops $a$ and $b$ have the same price, then $a \prec b$ and $b\prec a$ would hold.
So $\prec$ may not be antisymmetric.
\qed
\end{proof}

Suppose that under an additional constraint like CPU brand preference, the set of laptops becomes totally ordered. That is under a revised definition of $\prec$ of the form: $a \prec b$ if and only if price of laptop $a$ is less than that of laptop $b$ and if the prices are equal then CPU brand of $b$ must be preferred over $a$'s.

Suppose now that Alice has more knowledge about a subset $C$ of models in the set of laptops $\mathbb{S}$. Let these be labeled as \emph{crisp} and let the order on $C$ be $\prec_{|C}$. Using additional criteria, rough objects can be indicated. Though lower and upper approximations can be defined in the scenario, the granulations actually used are harder to arrive at without all the gory details.

This example once again shows that granulation and construction of approximations from granules may not be related to the construction of approximations from properties in a cumulative way. 
\end{example}

\section{Well Distribution of Objects Over Chains: WDC}
\begin{definition}
A distribution of rough objects relative to a chain of crisp objects $C$ will be said to be a \emph{Well Distribution of Objects over Chains} if the minimal assumptions (Subsection.\ref{sf}) and the following two conditions hold: 
\begin{enumerate}
 \item {$C$ forms a chain under $\prec$ order.}
 \item {$\varphi$ is an surjection onto $C^2 \setminus \Delta_C$ ($\Delta_C$ being the diagonal of $C$).}
\end{enumerate}
\end{definition}

In this case pairs of the form $(a, a)$ (with $a$ being crisp) are not permitted to be regarded as rough objects. This amounts to requiring clearer conditions on the idea of what rough objects ought to be. 

\begin{example}\label{exa3}
In the example for pre-well distributions, if Alice never let a crisp object be a rough object, then the resulting example would fall under well distribution of objects over chains. In other words, the laptops would be well distributed over the crisp objects (crisp models of laptops).
\end{example}

\begin{theorem}
When the objects are well distributed over the crisp objects, then the number of crisp objects would be related to the total number of objects by the formula: \[n - k = k^2 - k\] 
So, it is necessary that $n$ be a perfect square
\end{theorem}
\begin{proof}
\begin{itemize}
\item {Under the assumptions, an object is either rough or is crisp.}
\item {The number of rough objects is $n - k$ .}
\item {By the nature of the surjection $n - k$ maps to $k^2 -k$  pairs of crisp objects (as the diagonal cannot represent rough objects).}
\item {So $n - k = k^2 - k$.}
\item {So $n= k^2$ is necessary.}  
\end{itemize}
\qed
\end{proof}

\begin{theorem}
Under the assumptions of this section, if the \shg is a Boolean algebra then the cardinality of the Boolean algebra $2^x$ is determined by integral solutions for $x$ in \[ 2^x = k^2 .\]  
\end{theorem}
\begin{proof}
As the number of elements in a finite power set must be of the form $2^x$ for some positive integer $x$, the correspondence follows. If $2^x = k^2$, then $x = {2 \log_2 k}$. \qed
\end{proof}

\begin{remark}
The previous theorem translates to a very sparse distribution of such  models. In fact for $n\leq 10^8$, the total number of models is $27$. \textsf{Fig.}\ref{fig3} gives an idea of the numbers that work.
\end{remark}

\begin{figure*}[hbt]
\centering
\includegraphics[width=11.7cm]{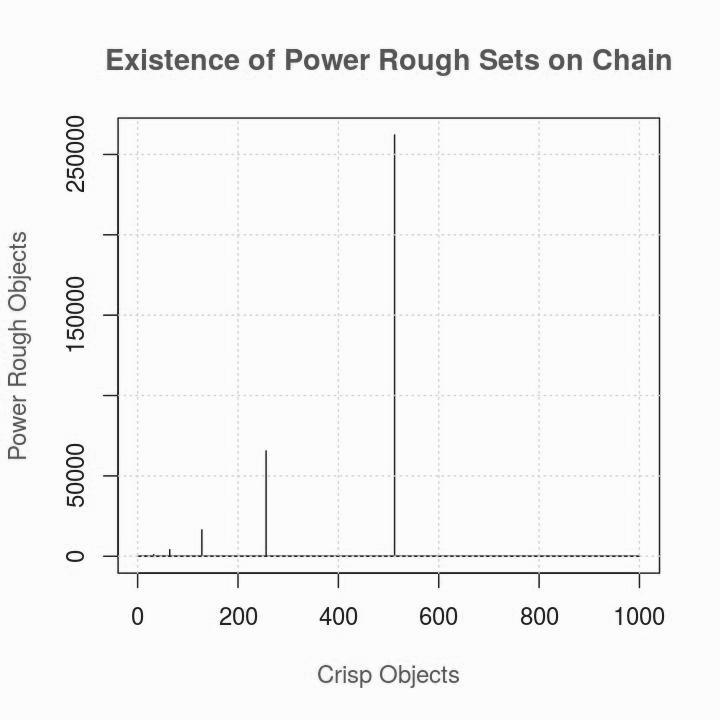}
\caption{Existence of Power Rough Sets on Chain}\label{fig3}
\end{figure*}

\section{Relaxed Distribution of Objects Over Chains: RDC}
\begin{definition}
A distribution of rough objects relative to a chain of crisp objects $C$ will be said to be a $\alpha$-\emph{Relaxed Distribution of Objects over Chains} if the minimal assumptions (Subsection.\ref{sf}) and the following three conditions hold.
\begin{itemize}
\item {$C$ forms a chain under $\prec$ order.}
\item {$\varphi $ is not necessarily a surjection and}
\item {\[\# (\varphi(R)) \leq \alpha (k^2 - k),\] for some rational $\alpha\in (0, 1]$ (the interpretation of $\alpha$ being that of a loose upper bound rather than an exact one).}
\end{itemize}
Any value of $\alpha$ that is consistent with the inequality will be referred to as an \emph{admissible} value of $\alpha$. 
\end{definition}

\begin{example}
The following modifications, in the context of Example. \ref{exa3}, are more common in practice:
\begin{itemize}
\item {No non crisp laptops may be represented by some pairs of crisp laptops and consequently $\varphi$ would not be a surjection onto $C^2 \setminus \Delta_C$ and}
\item {an estimate of the number of rough laptops may be known (this applies when too many models are available).}
\end{itemize}
These can lead to some estimate of $\alpha$. \emph{It should be noted that a natural subproblem is that of finding good values of $\alpha$}.
\end{example}

\begin{theorem}\label{fract}
In the context of relaxed distribution of objects over chains it is provable that,
for fixed $n$ the possible values of $k$ correspond to integral solutions of the formula:
\[k = \dfrac{(\pi - 1) + \sqrt{(1-\pi)^2 +4n\pi}}{2\pi},\] subject to $k\leq \lfloor\sqrt{n}\rfloor$, $\# (\varphi(R)) = \pi (k^2 - k)$ and $0 < \pi \leq \alpha$.
\end{theorem}
\begin{proof}
\begin{itemize}
\item {When $n-k = \pi (k^2- k )$ then $\pi = \dfrac{(n - k)}{(k^2 - k)}$}
\item {So positive integral solutions of $k = \dfrac{(\pi - 1) + \sqrt{(1-\pi)^2 +4n\pi}}{2\pi}$ may be admissible.}
\item {The expression for $\alpha$ means that it can only take a finite set of values given $n$ as possible values of $k$ must be in the set $\{2, 3, \ldots , \lfloor\sqrt {\frac{n}{\alpha}}\rfloor\}$.}
\end{itemize}
\qed
\end{proof}

\begin{remark}
The bounds for $k$ are not necessarily the best ones.
\end{remark}

\begin{theorem}
In the proof of the above theorem (Thm. \ref{fract}), fixed values of $n$ and $\pi$ do not in general correspond to unique values of $k$ and unique models. 
\end{theorem}

If the mentioned bounds on $k$ are not imposed then it might appear that \emph{for $\pi=0.5$ and $n=1000000$, the number of values of $k$ that work seem to be $1413$}. If the bounds on $k$ are imposed then \textsf{Fig.}\ref{fig5} gives a description of the resulting pattern of values:

\begin{figure*}[hbt]
\centering
\includegraphics[width=11.7cm]{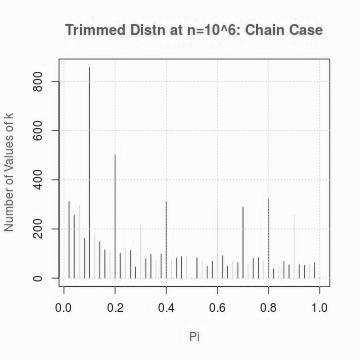}
\caption{Trimmed Number of Possible Values of $k$}\label{fig5}
\end{figure*}

\subsection*{Algorithms: RDC}
A purely arbitrary method of supplying values of $\alpha$ based on \emph{some heuristics} cannot be a tractable idea. To improve on this some algorithms for computing admissible values of $alpha$ are proposed in this subsection.

\begin{flushleft}
\textbf{RDC Algorithm-1} 
\end{flushleft}

\begin{enumerate}
\item {Fix the value of $n$.} 
\item {Start from possible values of $k$ less than $\sqrt{n - 1}$.}
\item {Compute $\alpha$ for all of these values.}
\item {Suppose the computed values are $\alpha_1, \ldots \alpha_r$}
\item {Check the admissibility of solutions.}
\end{enumerate}

\begin{flushleft}
\textbf{RDC Algorithm-2} 
\end{flushleft}

Another algorithm for converging to solutions is the following:
\begin{enumerate}
\item {Start from a sequence $\{\alpha_i\}$ of possible values in the interval $(0,1)$.}
\item {Check the admissibility and closeness to solutions}
\item {If a solution appears to be between $\alpha_i$ and $\alpha_{i+1}$, add an equally spaced subsequence between the two.}
\item {Check the admissibility and closeness to solutions.}
\item {Continue}
\item {Stop when solution is found}
\end{enumerate}

\begin{theorem}
Both of the above algorithms converge in a finite number of steps. 
\end{theorem}

\begin{proof}
Convergence of the first algorithm is obvious.

Convergence of the second follows from the following construction:
\begin{itemize}
\item {Suppose the goal is to converge to an $\alpha \in (0,1) $.}
\item {Let $\alpha_o =0, \, \alpha_1 = 1 $ and for a fixed positive integer $n$ and $i= 1, \ldots, n$, let $\alpha_{1i} = \frac{i}{n}$ and $\alpha \in (\alpha_{1j}, \alpha_{1 j+1})$. }
\item {Form $n$ number of equally spaced partitions $\{\alpha_{2i}\}$ of $(\alpha_{1j}, \alpha_{1 j+1})$ and let $\alpha \in (\alpha_{2j}, \alpha_{2 j+1})$.}
\item {Clearly $(\forall \epsilon >0 \, \exists N\, \forall r>N)\, |\alpha - \alpha_{rj}| < \epsilon $}
\item {So the algorithm will succeed in finding the required $\alpha$. }
\end{itemize}
\qed 
\end{proof}

\section{Relaxed Bounded Distribution on Chains: RBC}

\begin{definition}
A distribution of rough objects relative to a chain of crisp objects $C$ will be said to be a $\alpha$-\emph{Relaxed Bounded Distribution of Objects over Chains} (\textsf{RBC}) if the minimal assumptions (Subsection.\ref{sf}) and the following three conditions hold.

\begin{itemize}
\item {$C$ forms a chain under $\prec$ order and $\varphi $ is not necessarily a surjection,}
\item {\[\# (\varphi(R)) \leq \alpha (k^2 - k),\] for some rational $\alpha\in (0, 1]$ (the interpretation of $\alpha$ being that of a loose upper bound rather than an exact one) and}
\item{$R$ is partitioned into disjoint subsets of size $\{r_i \}_{i=1}^{g}$ with $g= k^2 -k$ subject to the condition 
\begin{align*}
\tag{$\beta$}  a \leq r_i \leq b \leq n -k, \text{  with  }a,\, b   \text{ being constants.}  
\end{align*}}
\end{itemize}
Any value of $\alpha$ that is consistent with the inequality will be referred to as an \emph{admissible} value of $\alpha$. 
\end{definition}

\textsf{RBC} differs from \textsf{RDC} in the explicit specification of bounds on number of objects that may be represented by a pair of crisp objects. 

\begin{proposition}
For $a=0$ and $b=n - k$, every \textsf{RDC} is a \textsf{RBC}. 
\end{proposition}

\begin{theorem}
If the crisp objects form a chain, then the total number of possible models $B$ is \[B \,= \sum_{\alpha\in \pi(r)|\beta}\prod_{i=1}^{k^2 -k} \alpha_i \text{  and  }n_o a^{k^2 - k} \leq B \leq n_o b^{k^2 - k} ,\]
with the summation being over partitions $\alpha = \{\alpha_i \}$ of $r$ subject to the condition $\beta$ and $n_o$ being the number of admissible partitions under the conditions.
\end{theorem}
\begin{proof}
\begin{itemize}
\item {On a chain of length $k$, $k^2 -k$ spaces can be filled.}
\item {The next step is to determine the partitions $\pi(r)$ of $r$ into $k^2 -k$ distinct parts.}
\item {The condition $\beta$ eliminates many of these partitions resulting in the admissible set of partitions $\pi(r)|\beta$.}
\item {Each of the partitions $\alpha \in \pi(r)|\beta$ corresponds to $\prod_i \alpha_i $ number of possibilities.}
\item {So the result follows.}
\end{itemize}
\qed
\end{proof}

\section{Distribution of Objects: General Context}

\begin{definition}
A distribution of rough objects relative to a poset of crisp objects $C$ will be said to be a $\alpha$-\emph{Relaxed Bounded Distribution of Objects } (\textsf{RBO}) if the minimal assumptions (Subsection.\ref{sf}) without the restriction \textsf{R2} and the following three conditions hold:
\begin{itemize}
\item {$\# (\varphi(R)) = t \leq n-k$,}
\item {$t = \beta (k^2 -k)$ and,}
\item {$n - k = \alpha (k^2 - k)$. for some constants $t,\, \beta,\, \alpha$}
\end{itemize}

\end{definition}

\begin{example}
In the context of Example.\ref{exa3}, if Alice is not able to indicate a single criteria for the chain order, then the whole context would naturally fall under the context of this section.  
\end{example}

This perspective can also be used in more general contexts that fall outside the scope of \textsf{SHG}. It is possible, in practice, that objects are neither crisp or rough. This can happen, for example, when:
\begin{itemize}
\item {a consistent method of identifying crisp objects is not used or}
\item {some objects are merely labeled on the basis of poorly defined partials of features or}
\item {a sufficiently rich set of features that can provide for consistent identification is not used}
\end{itemize}

For \textsf{RBO}, in the absence of additional information about the order structure, it is possible to rely on chain decompositions or use generalized ideals and choice functions for developing computational considerations based on the material of the earlier sections. The latter is specified first in what follows.

\begin{definition}
The \emph{lower definable scope} $\mathbf{SL}(x)$ of an element $x\in R$ will be the set of maximal elements in $\downarrow (x) \cap C$, that is \[\mathbf{SL}(x) = \max(\downarrow(x) \cap C). \]

The \emph{upper definable scope} $\mathbf{SU}(x)$ of an element $x\in R$ will be the set of minimal elements in $\uparrow (x) \cap C$, that is \[\mathbf{SU}(x) = \min(\uparrow(x) \cap C). \]
\end{definition}

All representations of rough objects can be seen as the result of choice operations \[\psi_x : \mathbf{SL}(x) \times \mathbf{SU}(x) \longmapsto C^2 \setminus \Delta_C . \] Letting $\# (\mathbf{SL} (x)) = c(x)$ and $\# (\mathbf{SU} (x)) = v(x)$ formulas for possible values may be obtainable. Finding a simplification without additional assumptions remains an open problem though.

\subsection*{Chain Covers}

Let $C^*$ be the set of crisp objects $C$ with the induced partial order, then by the theorem in Sec \ref{wth},  
the order structure of the poset of crisp objects $C^*$ permits a disjoint chain cover. This permits an incomplete strategy for estimating the structure of possible models and counting the number of models. 

\begin{itemize}
\item {Let $\{C_i \,:\, i=1, \ldots h \}$ be a disjoint chain cover of $C^*$. Chains starting from $a$ and ending at $b$ will be denoted by $[[a, b]]$.}
\item {Let $C_1$ be the chain $[[0, 1]]$ from the the smallest(empty) to the largest object.}
\item {If $C_1$ has no branching points, then without loss of generality, it can be assumed that $C_2 = [[c_{2l}, c_{2g}]]$ is another chain with least element $c_{2l}$ and greatest element $c_{2g}$ such that $0 \prec c_{21}$, possibly $c_{2g} \prec 1$ and certainly $c_{2g} < 1$. }
\item {If $c_{2g} < 1$, then the least element of at least two other chains ($[[c_{3l}, \, c_{3g}]]$ and $[[c_{4l}, \, c_{4g}]]$) must cover $c_{2g}$, that is $c_{2g}\prec c_{3l}$ and $c_{2g}\prec c_{4l}$.}
\item {This process can be extended till the whole poset is covered. }
\item {The first step for distributing the rough objects amongst these crisp objects consists in identifying the spaces distributed over maximal chains on the disjoint cover subject to avoiding over counting of parts of chains below branching points.}
\end{itemize}

The above motivates the following combinatorial problem for solving the general problem: 

Let $H = [[c_l,c_g]] $ be a chain of crisp objects with $\#(H) = \alpha$ and let $c_o$ be a branching point on the chain with $\# ([[c_l , c_o]]) = \alpha_o$. Let \[S_C = \{(a, b) \,; \,a, b\in [[c_o, c_g]] \text{ or } c_l < a, b < c_o \}.\] In how many ways can a subset $R_f \subseteq R$ of rough objects be distributed over $S_C$ under $\# (R_f) = \pi$?

\begin{theorem}
If the number of possible ways of distributing $r$ rough objects over a chain of $h$ crisp elements is $n(r,h )$, then the number of models in the above problem is \[n (r,h ) - n (r , h_o).\]  
\end{theorem}

\begin{proof}
This is because the places between crisp objects in $[[c_l,c_o ]] $ must be omitted. The exact expression of $n(r, h)$ has already been described earlier.
\qed 
\end{proof}

Using the above theorem it is possible to evaluate the models starting with splitting of $n-k$ into atmost $w$ partitions. \emph{Because of this it is not necessary to use principal order filters generated by crisp objects to arrive at direct counts of the number of possible cases and a representation schematics}. 

\subsection{Applications: Hybrid Swarm Optimization}

Many unsupervised and semi-supervised algorithms do not converge properly and steps involved may have dense and unclear meaning. The justification for using such algorithms often involve analogies that may appear to be reasonable at one level and definitely suspect in broader perspectives - typically this can be expected to happen when the independent intelligence of computational agents or potential sources of intelligence in the context are disregarded. For example, the class of ant colony algorithms (see \cite{DB2005}) uses probabilist assumptions and restricted scope for control at the cost of simplifying assumptions.

For example, for a set of robots to navigate unfamiliar terrain with obstacles, a swarm optimization method like the polymorphic ant colony optimization method may be used \cite{QCZ2014}. The method involves scouts, workers and other types of robots (ants). Additional information about the terrain can be used to assess the quality of paths being found through the methods developed in this paper - the guiding principle for this can be that \emph{if the approximations of obstacles or better paths do not fit in a rough scheme of things, then the polymorphic optimization method is warranted}.  

The other kind of situation where the same heuristics can apply is when the robots are not fully autonomous and under partial control as in a hacking context. More details of these applications will appear in a separate paper.
 
\section{Interpretation and Directions}

The results proved in this research are relevant from multiple perspectives. In the perspective that does not bother with issues of contamination, the results mean that the number of rough models relative to the number of other possible models of computational intelligence is low. This can be disputed as the signature of the model is restricted and categoricity does not hold. 

In the perspective of the contamination problem, the axiomatic approach to granules, the results help in handling inverse problems in particular. From a minimum of information, it may be possible to deduce
\begin{itemize}
\item {whether a rough model is possible or}
\item {whether a rough model is not possible or}
\item {whether the given data is part of some minimal rough extensions}
\end{itemize}
The last possibility can be solved by keeping fixed the number of rough objects or otherwise. These problems apply for the contaminated approach too. It should be noted that extensions need to make sense in the first place.
The results can also be expected to have many applications in hybrid, probabilist approaches and variants.

An important problem that has not been explored in this paper is the concept of isomorphism between higher granular operator spaces. This is considered in a forthcoming paper by the present author.

\bibliographystyle{splncs.bst}
\bibliography{../bib/biblioam2016xx.bib}

\begin{thebibliography}{10}

\bibitem{ZPB}
Pawlak, Z.:
\newblock {Rough Sets: Theoretical Aspects of Reasoning About Data}.
\newblock Kluwer Academic Publishers, Dodrecht (1991)

\bibitem{gdu}
Duntsch, I., Gediga, G.:
\newblock {Rough set data analysis: A road to non-invasive knowledge
  discovery}.
\newblock Methodos Publishers (2000)

\bibitem{AM240}
Mani, A.:
\newblock {Dialectics of Counting and the Mathematics of Vagueness}.
\newblock In Peters, J.F., Skowron, A., eds.: {Transactions on Rough Sets ,
  LNCS 7255}. Volume~XV.
\newblock Springer Verlag (2012)  122--180

\bibitem{lp2011}
Polkowski, L.:
\newblock {Approximate Reasoning by Parts}.
\newblock Springer Verlag (2011)

\bibitem{sk05}
Skowron, A.:
\newblock {Rough sets and vague concepts}.
\newblock Fund. Inform. \textbf{64}(1-4) (2005)  417--431

\bibitem{ppm2}
Pagliani, P., Chakraborty, M.:
\newblock {A Geometry of Approximation: Rough Set Theory: Logic, Algebra and
  Topology of Conceptual Patterns}.
\newblock Springer, Berlin (2008)

\bibitem{BC1}
Banerjee, M., Chakraborty, M.K.:
\newblock {Rough Sets through Algebraic Logic}.
\newblock Fundamenta Informaticae \textbf{28} (1996)  211--221

\bibitem{du}
Duntsch, I.:
\newblock {Rough Sets and Algebras of Relations}.
\newblock In Orlowska, E., ed.: {Incomplete Information and Rough Set
  Analysis}, Physica, Heidelberg (1998)  109--119

\bibitem{AM3}
Mani, A.:
\newblock {Super Rough Semantics}.
\newblock Fundamenta Informaticae \textbf{65}(3) (2005)  249--261

\bibitem{IGO2007}
Duntsch, I., Gediga, G., Orlowska, E.:
\newblock {Relational Attribute Systems II}.
\newblock In Peters, J.F., Skowron, A., Marek, V., Orlowska, E., Slowinski, R.,
  Ziarko, W., eds.: {Transactions on Rough Sets VII}. {LNCS 4400}.
\newblock Springer Verlag (2007)  16--35

\bibitem{IGO2001}
Duntsch, I., Gediga, G., Orlowska, E.:
\newblock {Relational Attribute Systems}.
\newblock International Journal of Human Computer Studies \textbf{55}(3) (2001)
   293--309

\bibitem{mal}
Malcev, A.I.:
\newblock {The Metamathematics of Algebraic Systems -- Collected Papers}.
\newblock North Holland (1971)

\bibitem{BC2}
Banerjee, M., Chakraborty, M.K.:
\newblock {Algebras from Rough Sets -- an Overview}.
\newblock In Pal, S.K., {et. al}, eds.: {Rough-Neural Computing}.
\newblock Springer Verlag (2004)  157--184

\bibitem{IE2013}
Duntsch, I., Orlowska, E.:
\newblock {Discrete Duality for Rough Relation Algebras}.
\newblock Fundamenta Informaticae \textbf{127} (2013)  35--47

\bibitem{AM3930}
Mani, A.:
\newblock {Ontology, Rough Y-Systems and Dependence}.
\newblock Internat. J of Comp. Sci. and Appl. \textbf{11}(2) (2014)  114--136
  Special Issue of IJCSA on Computational Intelligence.

\bibitem{AM6999}
Mani, A.:
\newblock {Antichain Based Semantics for Rough Sets}.
\newblock In Ciucci, D., Wang, G., Mitra, S., Wu, W., eds.: {RSKT 2015},
  Springer-Verlag (2015)  319--330

\bibitem{AM9114}
Mani, A.:
\newblock {Knowledge and Consequence in AC Semantics for General Rough Sets }.
\newblock In Wang, G., Skowron, A., Yao, Y., Slezak, D., eds.: {Thriving Rough
  Sets----10th Anniversary - Honoring Professor Zdzis{\l}aw Pawlak's Life and
  Legacy \& 35 years of Rough Sets,}. Volume 708 of {Studies in Computational
  Intelligence Series}.
\newblock Springer International Publishing (2017)  1--36

\bibitem{AM9699}
Mani, A.:
\newblock {On Deductive Systems of AC Semantics for Rough Sets}.
\newblock ArXiv. Math (1610.02634v1) (October 2016)  1--12

\bibitem{JS09}
Stepaniuk, J.:
\newblock {Rough -- Granular Computing in Knowledge Discovery and Data Mining}.
\newblock {Studies in Computational Intelligence,Volume 152}. Springer-Verlag
  (2009)

\bibitem{AM9501}
Mani, A.:
\newblock {Algebraic Semantics of Proto-Transitive Rough Sets}.
\newblock In Peters, J.F., Skowron, A., eds.: {Transactions on Rough Sets LNCS
  10020}. Volume~XX.
\newblock Springer Verlag (2016)  51--108

\bibitem{AM960}
Mani, A.:
\newblock {Towards an Algebraic Approach for Cover Based Rough Semantics and
  Combinations of Approximation Spaces}.
\newblock In Sakai, H.,  et~al., eds.: {RSFDGrC 2009}. Volume LNAI 5908.,
  Springer-Verlag (2009)  77--84

\bibitem{AM6900}
Mani, A.:
\newblock {Pure Rough Mereology and Counting}.
\newblock In: {WIECON,2016}, IEEXPlore (2016)  1--8

\bibitem{JGC2009}
Kung, J.P.S., Rota, G.C., Yan, C.H.:
\newblock {Combinatorics-The Rota Way}.
\newblock Cambridge University Press (2009)

\bibitem{DP2002}
Davey, B.A., Priestley, H.A.:
\newblock {Introduction to Lattices and Order}. Second edn.
\newblock Cambridge University Press (2002)

\bibitem{MHRL}
Morey, R.D., Hoekstra, R., Rouder, J.N., Lee, M.D.:
\newblock {The Fallacy of Placing Confidence in Confidence Intervals}.
\newblock Technical report (August 2015)

\bibitem{DB2005}
Dorigo, M., Blum, C.:
\newblock {Ant colony optimization theory: A survey}.
\newblock Theoretical Computer Science \textbf{344}(2-3) (2005)  243--278

\bibitem{QCZ2014}
Qing, C., Zhong, Y., Xiang, L.:
\newblock {Cloud database dynamic route scheduling based on polymorphic ant
  colony optimization algorithm}.
\newblock Computer Modelling and New Technologies \textbf{18}(2) (2014)
  161--165

\end{thebibliography}
\end{document}